\documentclass[10pt,a4paper]{amsart}
\usepackage[utf8]{inputenc}

\title{Segal Spaces, Spans, and Semicategories}
\author{Rune Haugseng}
\address{Norwegian University of Science and Technology (NTNU),
  Trondheim, Norway}
\urladdr{http://folk.ntnu.no/runegha/}

\usepackage{amsmath,amssymb,amsthm}
\usepackage[utf8]{inputenc}
\usepackage{url}
\usepackage[shortlabels]{enumitem}
  \setitemize[1]{leftmargin=2em}
  \setenumerate[1]{leftmargin=*}
\usepackage[colorlinks=true,linktocpage=true,citecolor=blue,unicode,hyperindex,breaklinks]{hyperref}
\usepackage[capitalise]{cleveref}
\crefformat{equation}{(#2#1#3)}
\usepackage{color}
\usepackage{tikz}
\usetikzlibrary{matrix,arrows}
\usepackage{tikz-cd}
\usepackage[alphabetic]{amsrefs}
\usepackage{xstring}
\renewcommand{\eprint}[1]{\IfBeginWith{#1}{arXiv}{\href{https://arxiv.org/abs/#1}{#1}}{\href{#1}{#1}}}

\usepackage{eucal}

\newcommand{\bbDelta}{\boldsymbol{\Delta}}
\newcommand{\bbSigma}{\boldsymbol{\Sigma}}

\theoremstyle{theorem}

\newtheorem{thm}{Theorem}[section]
\newtheorem{lemma}[thm]{Lemma}
\newtheorem{propn}[thm]{Proposition}
\newtheorem{cor}[thm]{Corollary}

\newtheorem*{thm*}{Theorem}
\newtheorem*{conjecture*}{Conjecture}
\newtheorem*{goal*}{Goal}
\newtheorem*{question*}{Question}
\newtheorem*{prethm*}{Pretheorem}
\theoremstyle{definition}
\newtheorem{defn}[thm]{Definition}

\newtheorem{notation}[thm]{Notation}

\newtheorem{warning}[thm]{Warning}

\newtheorem{remark}[thm]{Remark}

\theoremstyle{remark}

\newcommand{\txt}[1]{\ensuremath{\text{\textup{#1}}}}
\newcommand{\catname}[1]{\txt{#1}}
\newcommand{\Set}{\catname{Set}}

\newcommand{\Cat}{\catname{Cat}}
\newcommand{\CatI}{\catname{Cat}_\infty}

\newcommand{\Alg}{\catname{Alg}}
\newcommand{\Opd}{\catname{Opd}}
\newcommand{\OpdI}{\Opd_{\infty}}

\newcommand{\Fun}{\txt{Fun}}
\newcommand{\Map}{\txt{Map}}

\newcommand{\op}{\txt{op}}
\DeclareMathOperator{\Seg}{Seg}
\DeclareMathOperator{\Span}{Span}
\DeclareMathOperator{\SPAN}{SPAN}

\newcommand{\IFF}{if and only if}
\newcommand{\icat}{$\infty$-category}
\newcommand{\icats}{$\infty$-categories}
\newcommand{\icatl}{$\infty$-categorical}

\newcommand{\iopd}{$\infty$-operad}
\newcommand{\iopds}{$\infty$-operads}
\newcommand{\nsiopd}{non-symmetric $\infty$-operad}
\newcommand{\nsiopds}{non-symmetric $\infty$-operads}
\newcommand{\gnsiopd}{generalized non-symmetric $\infty$-operad}
\newcommand{\gnsiopds}{generalized non-symmetric $\infty$-operads}

\newcommand{\ie}{i.e.\@}

\newcommand{\isoto}{\xrightarrow{\sim}}
\newcommand{\isofrom}{\xleftarrow{\sim}}

\newcommand{\xto}[1]{\xrightarrow{#1}}

\newcommand{\from}{\leftarrow}
\newcommand{\xfrom}[1]{\xleftarrow{#1}}

\newcommand{\blank}{\text{\textendash}}
\newcommand{\id}{\txt{id}}

\newcommand{\simp}{\bbDelta}
\newcommand{\Dop}{\simp^{\op}}

\title{Segal Spaces, Spans, and Semicategories}

\date{\today. This paper was written while the author was employed at
  the IBS Center for Geometry and Physics in a position funded by 
  grant IBS-R003-D1 of the Institute for Basic Science, Republic of Korea.}

\newcommand{\oSPAN}{\overline{\SPAN}}
\newcommand{\SPANp}{\vphantom{\SPAN}\smash{\SPAN}^{+}\!}
\newcommand{\Spanp}{\Span^{+}}
\newcommand{\oSPANp}{\vphantom{\SPAN}\smash{\oSPAN}^{+}\!}
\newcommand{\bbS}{\bbSigma}
\newcommand{\hbbS}{\widehat{\bbS}}
\newcommand{\Dinj}{\simp_{\txt{inj}}}
\newcommand{\Dinjop}{\Dop_{\txt{inj}}}
\newcommand{\Algnu}{\Alg_{\txt{nu}}}
\newcommand{\Algqu}{\Alg_{\txt{qu}}}
\newcommand{\Catnu}{\Cat_{\txt{nu}}}
\newcommand{\Catqu}{\Cat_{\txt{qu}}}
\newcommand{\Segnu}{\Seg_{\txt{nu}}}
\newcommand{\Segqu}{\Seg_{\txt{qu}}}
\newcommand{\CSS}{\txt{CSS}}
\newcommand{\CSSqu}{\CSS_{\txt{qu}}}

\begin{document}

\begin{abstract}
  We show that Segal spaces, and more generally category objects in an
  \icat{} $\mathcal{C}$, can be identified with associative algebras
  in the double $\infty$-category of spans in $\mathcal{C}$. We use
  this observation to prove that ``having identities'' is a property
  of a non-unital $(\infty,n)$-category.
\end{abstract}

\maketitle

\tableofcontents

\section{Introduction}
A ``semicategory'' or non-unital category is a category without
identity morphisms. It is an easy exercise to show that ``having
identities'' is a property of a semicategory, and ``preserving
identities'' is a property of functors of semicategories. More
precisely, the forgetful functor $\Cat \to \txt{Semicat}$ gives an
equivalence between $\Cat$ and a subcategory of $\txt{Semicat}$.

The analogues of this statement for higher categories turn out to be
very useful: To define particular examples of higher categories or
functors between them, it can be extremely convenient to first ignore
the identities and then at the end check that the resulting non-unital
structure has the required property. For $(\infty,1)$-categories
(which we will refer to as \icats{}), such a result is already known:
it is due to Harpaz~\cite{HarpazQUnital} in the context of Segal
spaces\footnote{See \cref{warn:Harpaz} for the precise relation of
  our result to that of Harpaz.}, and for quasicategories it is a combination of work of
Tanaka~\cite{TanakaUnital} and Steimle~\cite{SteimleDeg}. Similar
results have also been proved for other higher-categorical structures,
including $A_{\infty}$-categories (see \cite{LyubashenkoManzyukUnital}
for a comparison of different notions of weak units in this setting)
and monoidal 2-categories~\cite{JoyalKockUnits}.

The goal of the present paper is to show that ``having identities'' is
also a property of $(\infty,n)$-categories for all $n$:
\begin{thm}
  Let $\Seg^{n}(\mathcal{S})$ denote the \icat{} of $n$-fold Segal
  spaces and $\Segnu^{n}(\mathcal{S})$ its non-unital analogue. Then
  the forgetful functor $\Seg^{n}(\mathcal{S}) \to
  \Segnu^{n}(\mathcal{S})$ induces an equivalence
  \[ \Seg^{n}(\mathcal{S}) \isoto \Segqu^{n}(\mathcal{S}) \] where
  $\Segqu^{n}(\mathcal{S}) \subseteq \Segnu^{n}(\mathcal{S})$ is a
  subcategory of \emph{quasi-unital} $n$-fold Segal spaces and
  quasi-unital functors between them.
\end{thm}
We will prove this in \cref{sec:qucat} by first proving the case
$n = 1$ for category objects (or internal \icats{}) in any \icat{}
with finite limits; the general statement then follows easily by
iterating this case.

In the case $n = 1$ we will deduce the theorem from the analogous
statement for non-unital associative algebras in monoidal \icats{},
which has been proved by Lurie~\cite{HA}. To do so, we must first identify
category objects as certain associative algebras. For ordinary
categories, it seems to have been first observed by
B\'enabou~\cite{BenabouBicat} that a category can be viewed as an associative
algebra (or monad) in a 2-category $\Spanp(\Set)$ of spans of sets;
this has
\begin{itemize}
\item sets as objects,
\item spans $I \from {S} \to J$ as 1-morphisms from $I$ to $J$, with
  composition given by taking pullbacks, \ie{}
  \[
  \begin{tikzcd}[cramped,row sep=small,column sep=tiny]
    {} & S \arrow{dl} \arrow{dr} \\
    J & & K
  \end{tikzcd}
  \!\circ\!
  \begin{tikzcd}[cramped,row sep=small,column sep=tiny]
    {} & T \arrow{dl} \arrow{dr} \\
    I & & J
  \end{tikzcd}
  :=
  \begin{tikzcd}[cramped,row sep=small,column sep=tiny]
    {} & T \times_{J}S \arrow{dl} \arrow{dr} \\
    I & & K,
  \end{tikzcd}
\]
\item and morphisms of spans
  \[
    \begin{tikzcd}[column sep=small]
      {} & S \arrow{dl} \arrow{drr} \arrow{r} &  S' \arrow[crossing
      over]{dll} \arrow{dr} \\
      I & & & J
    \end{tikzcd}
  \]
  as 2-morphisms, composing in the obvious way.
\end{itemize}
In particular, a category with $S$ as its set of objects is the same
thing as an associative algebra in the ``double slice'' $\Set_{/S,S}$
with the tensor product defined by pullbacks over $S$. However,
functors are not the same thing as morphisms of algebras in
$\Spanp(\Set)$. To remedy this, we can upgrade to a double category
$\SPANp(\Set)$ whose objects are sets, vertical morphisms are
functions, horizontal morphisms are spans, and whose squares are
diagrams of the form
\[
  \begin{tikzcd}[sep=small]
    \bullet \arrow{d} & \bullet \arrow{d} \arrow{r} \arrow{l} & \bullet\arrow{d} \\
    \bullet & \bullet  \arrow{r} \arrow{l} & \bullet.
  \end{tikzcd}
\]
(This example of a double category is discussed in some detail in
\cite{GrandisPareLimits}*{\S 3.2}; the earliest reference is perhaps
\cite{BurroniTCat}*{Remarque on p.~294}.) We can then consider
associative algebras (also known as monoids or monads) in the double
category $\SPANp(\Set)$, which are the same thing as algebras in its
horizontal 2-category and so again give categories. However, the
vertical morphisms give a new notion of morphisms of algebras, which
in this case recovers functors betwen categories; thus $\Cat$ is
equivalent to the category of associative algebras in
$\SPANp(\Set)$. (This observation can be found in
\cite{LeinsterHigherOpds}*{Example 5.3.5},
\cite{ShulmanFramed}*{Example 11.2}, and
\cite{FioreGambinoKock}*{Example 2.6}.) In \cref{sec:span} we prove an
\icatl{} version of this statement, using the double \icat{} of spans
constructed in \cite{spans}:
\begin{thm}
  Let $\mathcal{C}$ be an \icat{} with finite limits. There is an
  equivalence of \icats{}
  \[ \Cat(\mathcal{C}) \simeq \Alg(\SPANp(\mathcal{C}))\]
  between category objects in $\mathcal{C}$ and associative algebras
  in the double \icat{} of spans in $\mathcal{C}$.
\end{thm}

\section{Preliminaries}\label{sec:prelim}
In this section we briefly review the higher-algebraic structures we
will make use of below.

\begin{notation}
  We write $\simp$ for the simplex category of ordered sets $[n] :=
  \{0,\ldots,n\}$. A morphism $\phi \colon [m] \to [n]$ is called
  \emph{inert} if it is the inclusion of a subinterval, \ie{} $\phi(i)
  = \phi(0)+i$ for $i = 0,\ldots,n$. For $0 \leq i \leq j \leq n$ we
  write $\rho_{i,j} \colon [j-i] \to [n]$ for the inert morphism with
  $\rho_{i,j}(0) = i$, $\rho_{i,j}(j-i)= j$.
\end{notation}

\begin{defn}
  Let $\mathcal{C}$ be an \icat{} with pullbacks. A \emph{category
    object} in $\mathcal{C}$ is a simplicial object $X \colon \Dop \to
  \mathcal{C}$ such that for all $[n] \in \simp$ the morphism
  \[ X_{n}\to X_{1} \times_{X_{0}} \cdots \times_{X_{0}} X_{1},\]
  induced by the morphisms $\rho_{i,(i+1)}$ and $\rho_{i,i}$, is an
  equivalence. We write $\Cat(\mathcal{C})$ for the full subcategory
  of $\Fun(\Dop, \mathcal{C})$ spanned by the category
  objects. Category objects in the \icat{} $\mathcal{S}$ of spaces are
  called \emph{Segal spaces} \cite{RezkCSS}.
\end{defn}

\begin{remark}
  Category objects in $\mathcal{C}$ model the algebraic structure of a
  (homotopy-coherent) category internal to $\mathcal{C}$: If we think of $X_{0}$ as the objects
  of $X$ and $X_{1}$ as the morphisms then we have:
  \begin{itemize}
  \item $d_{1},d_{0} \colon X_{1} \to
    X_{0}$, assigning source and target objects to morphisms,
  \item $s_{0} \colon X_{0} \to X_{1}$, assigning identity morphisms
    to objects,
  \item $X_{1} \times_{X_{0}} X_{1} \isofrom X_{2} \xto{d_{1}} X_{1}$,
    assigning composites to composable pairs of morphisms.
  \end{itemize}
  The remaining structure in the simplicial object $X$ ensures that
  the composition is homotopy-coherently associative and unital.
\end{remark}

\begin{remark}
  If $\mathcal{C}$ has a terminal object $*$, we can identify the
  category objects $X \colon \Dop \to \mathcal{C}$ such that $X_{0}
  \simeq *$ with \emph{associative monoids} in $\mathcal{C}$.
\end{remark}

\begin{defn}
  A \emph{double \icat{}} is a cocartesian fibration $\mathcal{M} \to
  \Dop$ such that the corresponding functor $\Dop \to \CatI$ is a
  category object. A \emph{monoidal \icat{}} is a double \icat{}
  $\mathcal{M}$ such that $\mathcal{M}_{0}$ is contractible,
  corresponding to an associative monoid in $\CatI$.
\end{defn}

A double \icat{} is thus an \icatl{} analogue of a category internal
to categories, or a double category. This notion has a useful
generalization:
\begin{defn}
  A \emph{generalized non-symmetric \iopd{}} is a functor $ p \colon \mathcal{O}
  \to \Dop$ such that:
    \begin{enumerate}[(i)]
  \item For every object $X$ in $\mathcal{O}_{n} := \mathcal{O}
    \times_{\Dop} \{[n]\}$, and every inert morphism $\phi \colon [m]
    \to [n]$ in $\simp$, there exists a $p$-cocartesian morphism
    $X \to \phi_{!}X$ lying over $\phi$.
  \item For every object $[n] \in \simp$, the functor
    \[ \mathcal{O}_{n} \to \mathcal{O}_{1} \times_{\mathcal{O}_{0}}
      \cdots \times_{\mathcal{O}_{0}} \mathcal{O}_{1} \]
    induced by the cocartesian morphisms over the maps $\rho_{i,i+1}$
    and $\rho_{i,i}$, is an
    equivalence.
  \item Given $X$ in $\mathcal{O}_{n}$, choose compatible cocartesian
    lifts $X \to X_{i,j}$ over $\rho_{i,j}$.
    Then for any $Y \in
    \mathcal{O}_{m}$, the
    commutative square
    \[
      \begin{tikzcd}
        \Map_{\mathcal{O}}(Y, X) \arrow{r} \arrow{d} &
        \Map_{\mathcal{O}}(Y, X_{0,1})
        \times_{\Map_{\mathcal{O}}(Y,X_{1,1})} \cdots
        \times_{\Map_{\mathcal{O}}(Y, X_{n-1,n-1})}
        \Map_{\mathcal{O}}(Y, X_{n-1,n}) \arrow{d} \\
        \Map_{\Dop}([m], [n]) \arrow{r} & \Map_{\Dop}([m], [1])
        \times_{\Map_{\Dop}([m],[0])} \cdots\times_{\Map_{\Dop}([m],[0])}\Map_{\Dop}([m], [1])
      \end{tikzcd}
      \]
      is cartesian.
  \end{enumerate}
\end{defn}

% \begin{remark}
%   We only recall this definition because it is the natural setting for
%   \cref{AlgDopC} and \cref{Seg0cart}; the only (generalized) \nsiopds{} that will actually
%   concern us in this paper are $\Dop$, which describes associative
%   algebras, and
%   {}
% \end{remark}

\begin{remark}
  Generalized non-symmetric \iopds{} are an \icatl{} analogue of the
  \emph{virtual double categories} of \cite{CruttwellShulman} or
  \textbf{fc}\emph{-multicategories} of \cite{LeinsterGenEnr}; see
  \cite{enr}*{\S 2} for further discussion and motivation. We can
  identify the double \icats{} as the generalized non-symmetric
  \iopds{} that are cocartesian fibrations.
\end{remark}

\begin{defn}
  A \emph{non-symmetric \iopd{}} is a generalized non-symmetric
  \iopd{} $\mathcal{O}$ such that $\mathcal{O}_{0}$ is contractible.
\end{defn}

\begin{defn}
  Let $\mathcal{O}$ be a generalized non-symmetric \iopd{}, and
  suppose $\mathcal{C}$ is an \icat{} with pullbacks. A \emph{Segal
    $\mathcal{O}$-object} in $\mathcal{C}$ is a functor $\Phi \colon
  \mathcal{O} \to \mathcal{C}$ such that for all $X \in
  \mathcal{O}_{n}$ the natural map
  \[ \Phi(X) \to \Phi(X_{0,1}) \times_{\Phi(X_{1,1})} \cdots
    \times_{\Phi(X_{n-1,n-1})} \Phi(X_{n-1,n})\]
  is an equivalence, where $X \to X_{i,j}$ is a cocartesian morphism
  over $\rho_{i,j}$. We write $\Seg_{\mathcal{O}}(\mathcal{C})$ for
  the full subcategory of $\Fun(\mathcal{O},\mathcal{C})$ spanned by
  the Segal $\mathcal{O}$-objects.
\end{defn}

%%%%

\begin{lemma}\label{i0RKE}
  Let $\mathcal{C}$ be an \icat{} with finite limits and $\mathcal{O}$
  a \gnsiopd{}, and let $i_{\mathcal{O}}$ denote the inclusion
  $\mathcal{O}_{0} \hookrightarrow \mathcal{O}$. Then right Kan
  extension along $i_{\mathcal{O}}$ gives a functor
  \[ i_{\mathcal{O},*} \colon \Fun(\mathcal{O}_{0}, \mathcal{C}) \to
    \Seg_{\mathcal{O}}(\mathcal{C}), \]
  right adjoint to  the restriction $i_{\mathcal{O}}^{*} \colon
  \Seg_{\mathcal{O}}(\mathcal{C}) \to \Fun(\mathcal{O}_{0}, \mathcal{C})$.
\end{lemma}
\begin{proof}
  For $F \colon \mathcal{O}_{0} \to \mathcal{C}$ and
  $X \in \mathcal{O}_{n}$, we have
  \[i_{\mathcal{O},*}F(X) \simeq \lim_{(Y,X\to Y) \in
      \mathcal{O}_{0}\times_{\mathcal{O}} \mathcal{O}_{X/}} F(Y),\]
  provided this limit exists in $\mathcal{C}$.
  If $X \to X_{i,i}$ is a cocartesian morphism over $\rho_{i,i} \colon
  [0] \to [n]$ then the discrete set $\{X \to X_{i,i}\}$ is a coinitial
  subcategory of $\mathcal{O}_{0}\times_{\mathcal{O}}
  \mathcal{O}_{X/}$, and so we have
  \[ i_{\mathcal{O},*}F(X) \simeq \prod_{i = 0}^{n} F(X_{i,i}),\]
  which exists provided $\mathcal{C}$ has finite limits and clearly
  gives a Segal object. Since the right Kan extension
  $i_{\mathcal{O},*}$ is the right adjoint to
  $i_{\mathcal{O}}^{*} \colon \Fun(\mathcal{O}, \mathcal{C}) \to
  \Fun(\mathcal{O}_{0}, \mathcal{C})$, it follows that the adjunction
  restricts to the full subcategory $\Seg_{\mathcal{O}}(\mathcal{C})$.
\end{proof}

\begin{remark}
  The object $i_{\mathcal{O},*}F$ is the terminal Segal
  $\mathcal{O}$-object whose restriction to $\mathcal{O}_{0}$ is
  $F$. For $\mathcal{C} = \mathcal{S}$, we can think of the object
  $i_{\mathcal{O},*}F(X) \simeq \prod_{i=0}^{n} F(X_{i,i})$ as
  the space of ``labels'' over $\mathcal{O}_{0}$ that a Segal
  $\mathcal{O}$-object $\Phi$ in $\mathcal{S}$ with
  $\Phi|_{\mathcal{O}_{0}} \simeq F$ would assign to points of
  $\Phi(X)$ (with the assignment of such given by the unit map
  $\Phi \to i_{\mathcal{O},*}F$). For example, when $\mathcal{O}$ is
  $\Dop$ (where $(\Dop)_{0}$ is just a point) we have for
  $X \in \mathcal{S} \simeq \Fun((\Dop)_{0}, \mathcal{S})$ that $(i_{\Dop,*}X)_{n} \simeq X^{\times n+1}$;
  if $\Phi \colon \Dop \to \mathcal{S}$ is a Segal space with
  $\Phi_{0} \simeq X$ then we can think of $\Phi_{n}$ as a space of
  strings of $n$ composable morphisms, and the unit map
  $\Phi_{n} \to (i_{\Dop,*}X)_{n} \simeq X^{n+1}$ assigns to such a
  string the list of $n+1$ objects appearing in it.
\end{remark}

\begin{propn}\label{Seg0cart}
  Let $\mathcal{C}$ be an \icat{} with finite limits and $\mathcal{O}$
  a \gnsiopd{}, and let $i_{\mathcal{O}}$ denote the inclusion
  $\mathcal{O}_{0} \hookrightarrow \mathcal{O}$. Then the functor
  \[i_{\mathcal{O}}^{*} \colon \Seg_{\mathcal{O}}(\mathcal{C}) \to
  \Fun(\mathcal{O}_{0}, \mathcal{C})\] is a cartesian fibration. For
  $\Phi \in \Seg_{\mathcal{O}}(\mathcal{C})$ and
  $\eta \colon F \to i_{\mathcal{O}}^{*}\Phi$, the cartesian morphism over
  $\eta$ is given by the pullback
  \[
    \begin{tikzcd}
      \eta^{*}\Phi \arrow{r} \arrow{d} & \Phi \arrow{d} \\
      i_{\mathcal{O},*}F \arrow{r}{i_{\mathcal{O},*}\eta} & i_{\mathcal{O},*}i_{\mathcal{O}}^{*}\Phi
    \end{tikzcd}
  \]
  in $\Fun(\mathcal{O}, \mathcal{C})$.
\end{propn}
\begin{proof}
  The functor
  $i_{\mathcal{O}}^{*} \colon \Seg_{\mathcal{O}}(\mathcal{C}) \to
  \Fun(\mathcal{O}_{0},\mathcal{C})$ has a right adjoint by \cref{i0RKE}.
  To see that $i_{\mathcal{O}}^{*}$ is a cartesian fibration, we apply
  the criterion of \cite[Corollary 4.52]{nmorita}. We must check that
  for $\Phi \in \Seg_{\mathcal{O}}(\mathcal{C})$ and
  $\eta \colon F \to i^{*}_{\mathcal{O}}\Phi$, if we define
  $\eta^{*}\Phi$ by the pullback square above, then the composite
  $i_{\mathcal{O}}^{*}\eta^{*}\Phi \to
  i_{\mathcal{O}}^{*}i_{\mathcal{O},*}F \to F$ is an
  equivalence.

  Since $i_{\mathcal{O}}$ is fully faithful we have
  $i^{*}_{\mathcal{O}}i_{\mathcal{O},*} \simeq \id$ and as
  $i_{\mathcal{O}}^{*}$ preserves limits we see that
  $i_{\mathcal{O}}^{*}\eta^{*}\Phi$ is the pullback
  \[
  \begin{tikzcd}
    i_{\mathcal{O}}^{*}\eta^{*}\Phi \arrow{r} \arrow{d} & i_{\mathcal{O}}^{*}\Phi \arrow[equals]{d} \\
    F \arrow{r} & i_{\mathcal{O}}^{*}\Phi,
  \end{tikzcd}
  \]
  whence $i_{\mathcal{O}}^{*}\eta^{*}\Phi \to F$ is indeed an
  equivalence. The characterization of cartesian morphisms follows
  from \cite[Proposition 4.51]{nmorita}.
\end{proof}

\begin{defn}
  A morphism of \gnsiopds{} is a commutative triangle
  \[
    \begin{tikzcd}
      \mathcal{O} \arrow{dr} \arrow{rr}{f} & & \mathcal{P} \arrow{dl} \\
       & \Dop,
    \end{tikzcd}
  \]
  where $f$ preserves cocartesian morphisms over inert maps in
  $\Dop$. We write $\OpdI^{\txt{gns}}$ for the \icat{} of \gnsiopds{},
  defined as a subcategory of $\Cat_{\infty/\Dop}$. A morphism of
  \gnsiopds{} from $\mathcal{O}$ to $\mathcal{P}$ is also called an
  \emph{$\mathcal{O}$-algebra} in $\mathcal{P}$, and we write
  $\Alg_{\mathcal{O}}(\mathcal{P})$ for the \icat{} of
  $\mathcal{O}$-algebras in $\mathcal{P}$, defined as a full
  subcategory of $\Fun_{/\Dop}(\mathcal{O},\mathcal{P})$.
\end{defn}

\begin{defn}
  For the terminal (generalized) \nsiopd{} $\Dop$, we refer to
  $\Dop$-algebras in a \gnsiopd{} $\mathcal{O}$ as \emph{associative
    algebras}, and write \[\Alg(\mathcal{O}) := \Alg_{\Dop}(\mathcal{O}).\]
\end{defn}

\begin{defn}
  For $\mathcal{C}$ an \icat{}, we write $\Dop_{\mathcal{C}} \to \Dop$
  for the cocartesian fibration corresponding to the right Kan
  extension $i_{\Dop,*}\mathcal{C} \colon \Dop \to
  \CatI$. \cref{i0RKE} shows that $\Dop_{\mathcal{C}}$ is a double
  \icat{}.
\end{defn}

If $\mathcal{M} \to \Dop$ is any double \icat{} then there is a
canonical morphism $\mathcal{M} \to \Dop_{\mathcal{M}_{0}}$ over
$\Dop$, corresponding to the unit morphism
$M \to i_{\Dop,*}i_{\Dop}^{*}M$ where $M$ is the functor
$\Dop \to \CatI$ associated to $\mathcal{M}$. This preserves all
cocartesian morphisms, and so is in particular a morphism of
\gnsiopds{}.

\begin{lemma}\label{AlgDopC}
  For any \gnsiopd{} $\mathcal{O}$, we have a natural equivalence
  \[\Alg_{\mathcal{O}}(\Dop_{\mathcal{C}}) \simeq
    \Fun(\mathcal{O}_{0}, \mathcal{C}),\]
  given by restriction to the fibre over $[0]$.
\end{lemma}
\begin{proof}
  Since we can replace $\mathcal{O}$ by $\mathcal{O} \times
  \Delta^{n}$, it suffices to show that we have a natural equivalence
  of mapping spaces
  \[ \Map_{\OpdI^{\txt{gns}}}(\mathcal{O}, \Dop_{\mathcal{C}}) \simeq
    \Map_{\CatI}(\mathcal{O}_{0}, \mathcal{C}).\] The \gnsiopd{} $\mathcal{O}$
  has an enveloping double \icat{} (see \cite{nmorita}*{\S A.8})
  $\txt{Env}(\mathcal{O})$ with a morphism of \gnsiopds{}
  $\mathcal{O} \to \txt{Env}(\mathcal{O})$ such that composition with
  this gives an equivalence between $\mathcal{O}$-algebras in a double
  \icat{} $\mathcal{M}$ and functors $\txt{Env}(\mathcal{O}) \to
  \mathcal{M}$ that preserve all cocartesian morphisms. Moreover,
  $\txt{Env}(\mathcal{O})_{0} \simeq \mathcal{O}_{0}$. It therefore
  suffices to prove that the natural map
  \[ \Map_{\Cat_{\infty/\Dop}^{\txt{cocart}}}(\mathcal{M},
    \Dop_{\mathcal{C}}) \to \Map_{\CatI}(\mathcal{M}_{0}, \mathcal{C})\] is an
  equivalence, where $\mathcal{M}$ is now a double \icat{}. If $M$ is
  the corresponding functor $\Dop \to \CatI$, then we can rewrite this as
  \[
    \begin{split}
      \Map_{\Cat_{\infty/\Dop}^{\txt{cocart}}}(\mathcal{M},
    \Dop_{\mathcal{C}}) & \simeq \Map_{\Fun(\Dop, \CatI)}(M,
    i_{*}\mathcal{C}) \\ & \simeq \Map_{\CatI}(i^{*}M, \mathcal{C}) \\
    & \simeq
    \Map_{\CatI}(\mathcal{M}_{0}, \mathcal{C}),
    \end{split}
\]
  as required.
\end{proof}

\begin{remark}\label{DblMon}
  Let $\mathcal{M}$ be a double \icat{} and $X$ an object of
  $\mathcal{M}_{0}$. Then $X$ induces a functor $\Dop \to
  \Dop_{\mathcal{M}_{0}}$ over $\Dop$ (corresponding to the morphism
  of right Kan extensions $i_{\Dop,*}\{X\} \to
  i_{\Dop,*}\mathcal{M}_{0}$), which preserves cocartesian
  morphisms. We can then define a monoidal \icat{}
  $\mathcal{M}_{X}^{\otimes}$ as the pullback
  \[
    \begin{tikzcd}
      \mathcal{M}_{X}^{\otimes} \arrow{r} \arrow{d} & \mathcal{M}
      \arrow{d} \\
      \Dop \arrow{r} & \Dop_{\mathcal{M}_{0}}
    \end{tikzcd}
  \]
  of \icats{}, which is also a pullback of cocartesian fibrations over
  $\Dop$. If we think of objects of $\mathcal{M}_{1}$ as ``horizontal
  morphisms'' in the double \icat{}, then this gives a monoidal
  structure on the \icat{}
  $\mathcal{M}_{1}(X,X) :=
  \mathcal{M}_{1}\times_{\mathcal{M}_{0}\times \mathcal{M}_{0}}
  \{(X,X)\}$ of horizontal endomorphisms of $X$, given by composition
  of horizontal morphisms.
\end{remark}

\section{Category Objects as Algebras in Spans}\label{sec:span}
In this section we will prove that category objects in an \icat{}
$\mathcal{C}$ can be identified with associative algebras in the
double \icat{} of spans in $\mathcal{C}$. We first recall the
construction of this double \icat{}, following \cite{spans} (see also
\cite{BarwickMackey, DyckerhoffKapranovTwoSeg, GaitsgoryRozenblyum1}
for alternative approaches).

\begin{defn}
  Let $\bbS^{n}$ denote the partially ordered set of pairs $(i,j)$
  with $0 \leq i \leq j \leq n$, where $(i,j) \leq (i',j')$ if
  $i \leq i' \leq j' \leq j$. These give a functor
  $\bbS^{\bullet} \colon \simp \to \Cat$, where for
  $\phi \colon [n] \to [m]$ the functor $\bbS^{n} \to \bbS^{m}$ takes
  $(i,j)$ to $(\phi(i),\phi(j))$. If $\mathcal{C}$ is an \icat{}, we
  write $\oSPANp(\mathcal{C}) \to \Dop$ for the cocartesian
  fibration corresponding to the functor
  \[\Fun(\bbS^{\bullet},\mathcal{C}) \colon \Dop \to \CatI.\]
  If $\mathcal{C}$ is an \icat{} with pullbacks, we write
  $\SPANp(\mathcal{C})$ for the full subcategory of
  $\oSPANp(\mathcal{C})$ spanned by the objects $F \colon \bbS^{n}
  \to \mathcal{C}$ such that the canonical morphism
  \[ F(i,j) \to F(i,i+1) \times_{F(i+1,i+1)} \cdots
    \times_{F(j-1,j-1)} F(j-1,j) \]
  is an equivalence for all $i,j$.
\end{defn}

\begin{propn}[\cite{spans}*{Proposition 5.14}]
  For any \icat{} $\mathcal{C}$ with pullbacks, the restricted functor
  $\SPANp(\mathcal{C}) \to \Dop$ is a double \icat{}.
\end{propn}

\begin{defn}
  For $X \in \mathcal{C}$, we can define a monoidal \icat{}
  $\mathcal{C}_{/X,X}^{\otimes} := \SPANp(\mathcal{C})_{X}$ as in
  \cref{DblMon}.  This gives a monoidal structure on the \icat{}
  $\mathcal{C}_{/X,X} \simeq \SPANp(\mathcal{C})_{1}(X,X)$ of
  objects of $\mathcal{C}$ equipped with two maps to $X$. The tensor
  product of $X \from Y \to X$ and $X \from Z \to X$ is defined by the
  pullback $X \from Y \times_{X} Z \to X$, and the unit is
  $X \xfrom{\id} X \xto{\id} X$.
\end{defn}

\begin{defn}
  Let $p \colon \hbbS \to \Dop$ denote the cartesian fibration for the
  functor $\bbS^{\bullet} \colon \simp \to \Cat$. We can identify
  objects of $\hbbS$ with pairs $([n], (i,j))$ where $[n] \in \simp$
  and $0 \leq i \leq j \leq n$; a morphism $([n], (i,j)) \to ([m],
  (i',j'))$ is given by a morphism $\phi \colon [m] \to [n]$ in
  $\simp$ such that $(i,j) \leq (\phi(i'), \phi(j'))$ in the partially
  ordered set $\bbS^{n}$. Note that this morphism is cartesian
  precisely when $(i,j) = (\phi(i'), \phi(j'))$. 
\end{defn}

\begin{propn}\label{SPANunivprop}
  For any \icat{} $\mathcal{I}$ over $\Dop$, there is a natural equivalence
  \[ \Fun_{/\Dop}(\mathcal{I}, \oSPANp(\mathcal{C})) \simeq
    \Fun(\mathcal{I} \times_{\Dop} \hbbS, \mathcal{C}).\] A functor
  $F \colon \mathcal{I} \to \oSPANp(\mathcal{C})$ takes a morphism
  $\phi \colon x \to y$ in $\mathcal{I}$ to a cocartesian morphism in
  $\oSPANp(\mathcal{C})$ \IFF{} the corresponding functor
  $F' \colon \mathcal{I} \times_{\Dop} \hbbS \to \mathcal{C}$ (which
  satisfies $F'(x, (i,j)) \simeq F(x)(i,j)$) takes all morphisms
  $(\phi, \gamma)$ where $\gamma$ in $\hbbS$ is cartesian to
  equivalences in $\mathcal{C}$.
\end{propn}
\begin{proof}
  The equivalence follows from \cite{freepres}*{Proposition 7.3},
  which identifies $\oSPANp(\mathcal{C})$ with the cocartesian
  fibration defined in \cite{HTT}*{Corollary 3.2.2.13(1)} by this
  universal property. The second statement then follows from the
  description of the cocartesian morphisms in \cite{HTT}*{Corollary
    3.2.2.13(2)}. 
\end{proof}

\begin{defn}
  We define a functor $\Pi \colon \hbbS \to \Dop$ by setting $\Pi([n],
  (i,j)) := [j-i]$ and sending a map $([n], (i,j)) \to ([m], (i',j'))$
  lying over $\phi \colon [m] \to [n]$ to the map $[j'-i'] \to [j-i]$
  given by $t \mapsto \phi(t+i')-i$. (In other words, we restrict
  $\phi$ to a map $\{i',i'+1,\ldots,j'\} \to \{i,i+1,\ldots,j\}$.)
\end{defn}

\begin{defn}
  We can define another functor $\Psi \colon \Dop \to \hbbS$ by
  sending $[n]$ to $([n], (0,n))$ and a morphism $\phi \colon [m] \to
  [n]$ in $\simp$ to the morphism $([n], (0,n)) \to ([m], (0,m))$
  lying over $\phi$. Observe that we have $\Pi\Psi \simeq
  \id_{\Dop}$. We can also define a natural transformation $\eta
  \colon \id_{\hbbS} \to \Psi\Pi$ by taking $\eta_{([n],(i,j))}$ to
  be the map $([n], (i,j)) \to ([j-i], (0,j-i))$ lying over $\rho_{i,j}$.
  We also have $p \Psi = \id_{\Dop}$, and a natural transformation
  $\epsilon \colon \Psi p \to \id_{\hbbS}$ given by the natural maps
  $([n],(0,n)) \to ([n],(i,j))$.
\end{defn}

\begin{propn}\label{PiLoc}
  The functor $\Pi \colon \hbbS \to \Dop$ exhibits $\Dop$ as the
  localization of $\hbbS$ at the set $I$ of cartesian morphisms that lie
  over inert maps in $\Dop$.
\end{propn}
\begin{proof}
  Let $W$ be the set of morphisms in $\hbbS$ that are mapped to
  isomorphisms (\ie{} identity morphisms) by $\Pi$. A morphism
  $([n],(i,j)) \to ([m],(i',j'))$ over $\phi \colon [m] \to [n]$ is in
  $W$ \IFF{} we have $j'-i'=j-i$ and $\phi(i'+t)=i+t$ for $t =
  0,\ldots,j'-i'$. In this case we have a commutative triangle
  \[
    \begin{tikzcd}
      ([n],(i,j))\arrow{dr} \arrow{rr} & & ([m],(i',j')) \arrow{dl} \\
      & ([j-i],(0,j-i))
    \end{tikzcd}
  \]
  where the two diagonal morphisms are in $I$. Thus by the 2-of-3
  property for equivalences, any functor $\hbbS \to \mathcal{C}$ that
  takes the maps in $I$ to equivalences takes all the maps in $W$ to
  equivalences. The localizations of $\hbbS$ at $I$ and $W$ are
  therefore the same. On the other hand, the components of the natural
  transformation $\eta$ are all in $W$, so using $\eta$ we see that composition
  with $\Psi$ is an inverse to
  \[ \Pi^{*} \colon \Fun(\Dop, \mathcal{C}) \to \Fun_{W}(\hbbS,
    \mathcal{C}),\] for any \icat{} $\mathcal{C}$, where
  $\Fun_{W}(\hbbS, \mathcal{C})$ denotes the full subcategory of
  functors $\hbbS \to \mathcal{C}$ that take the morphisms in $W$ to
  equivalences. In other words, $\Pi$ exhibits $\Dop$ as the
  localization of $\hbbS$ at $W$.
\end{proof}

\begin{propn}\label{intcocloc}
  Suppose $f \colon \mathcal{I} \to \Dop$ is a functor such that $\mathcal{I}$
  has $f$-cocartesian morphisms over inert maps in $\Dop$. Then there is a
  functor $\overline{\Pi} \colon \mathcal{I} \times_{\Dop} \hbbS
  \to \mathcal{I}$ (where the fibre product is over $p$) lying over
  $\Pi$, which exhibits $\mathcal{I}$ as the localization of
  $\mathcal{I} \times_{\Dop} \hbbS$ at the set $I_{\mathcal{I}}$ of morphisms $(x,(i,j))
  \to (x',(i',j'))$ such that $f(x) \to f(x')$ is inert, $I \to I'$ is
  cocartesian, and $(f(I),(i,j)) \to (f(I'),(i',j'))$ is cartesian.  
\end{propn}
\begin{proof}
  The functor $\Psi \colon \Dop \to \hbbS$ satisfies
  $p \circ \Psi \cong \id_{\Dop}$, and so induces a functor
  $\overline{\Psi} \colon \mathcal{I} \to \mathcal{I} \times_{\Dop}
  \hbbS$. If $\overline{p} \colon \mathcal{I} \times_{\Dop} \hbbS \to \mathcal{I}$ is
  the projection (which lies over $p \colon \hbbS \to \Dop$), then we
  have an equivalence $\overline{p} \overline{\Psi} \simeq \id$, and we also get a
  natural transformation $\overline{\epsilon} \colon \overline{\Psi} \overline{p} \to
  \id$ over $\epsilon$.
  
  Since the components of
  $\eta \colon \hbbS \times \Delta^{1} \to \hbbS$ lie over inert
  morphisms in $\Dop$, there is a unique cocartesian lift of $\eta$ to
  a natural transformation
  $\overline{\eta} \colon \mathcal{I} \times_{\Dop} \hbbS \to \mathcal{I}
  \times_{\Dop} \hbbS$, where $\overline{\eta}_{0}$ is the identity and
  $\overline{\eta}(x,(i,j))$ is $(x,(i,j)) \to (x_{i,j}, (0,j-i))$
  where $x \to x_{i,j}$ is a cocartesian morphism over
  $\rho_{i,j}$. (This follows from the lifting property obtained by
  combining \cite{HTT}*{Propositions 3.1.1.6 and 3.1.2.3}.) We define
  $\overline{\Pi} \colon \mathcal{I} \times_{\Dop} \hbbS \to \mathcal{I}$
  to be the composite $\overline{p}\overline{\eta}_{1}$. Then
  $\overline{\Pi}$ lies over $p\Psi\Pi$, which is $\Pi$ as
  $p \Psi = \id_{\Dop}$. We can identify $\overline{\eta}_{1}$ with
  $\overline{\Psi}\overline{\Pi}$ since
  \[\overline{\Psi} \overline{\Pi} \simeq \overline{\Psi}\overline{p} \overline{\eta}_{1}
    \xto{\overline{\epsilon}\overline{\eta}_{1}} \overline{\eta}_{1}\] is an
  equivalence. Moreover, $\overline{\eta}\overline{\Psi} \colon \overline{\Psi} \to
  \overline{\eta}_{1} \overline{\Psi}$ is an equivalence, being given by cocartesian
  morphisms over identities, hence $\overline{\Pi} \overline{\Psi} \simeq
  \overline{p}\overline{\Psi} \simeq \id$.
  
  Let $W_{\mathcal{I}}$ be the set of morphisms in
  $\mathcal{I} \times_{\Dop} \hbbS$ that are sent to equivalences by
  $\overline{\Pi}$. If $(x, (i,j)) \to (x',(i',j'))$ is such a
  morphism, then we have a commutative square
  \[
    \begin{tikzcd}
      (x, (i,j)) \arrow{r} \arrow{d} & (x',(i',j')) \arrow{d} \\
      (x_{i,j}, (0,j-i)) \arrow{r}{\sim} & (x'_{i',j'}, (0,j'-i')),
    \end{tikzcd}
    \]
  where the vertical maps are in $I_{\mathcal{I}}$. By the 2-of-3
  property for equivalences, this means that any functor that takes
  the morphisms in $I_{\mathcal{I}}$ to equivalences must take all
  morphisms in $W_{\mathcal{I}}$ to equivalences. Thus the
  localizations of $\mathcal{I} \times_{\Dop} \hbbS$ at
  $I_{\mathcal{I}}$ and $W_{\mathcal{I}}$ are the same. The same
  argument as in the proof of \cref{PiLoc} now shows that
  $\overline{\Pi}$ is exhibits $\mathcal{I}$ as the localization of
  $\mathcal{I} \times_{\Dop} \hbbS$ at $W_{\mathcal{I}}$.
\end{proof}

Combining \cref{SPANunivprop} with \cref{intcocloc}, we get:
\begin{cor}
  Suppose $f \colon \mathcal{I} \to \Dop$ is a functor such that
  $\mathcal{I}$ has $f$-cocartesian morphisms over inert maps in
  $\Dop$. Then there is a fully faithful functor of \icats{}
  \[ \Fun(\mathcal{I}, \mathcal{C}) \hookrightarrow
    \Fun_{\Dop}(\mathcal{I}, \oSPANp(\mathcal{C})) \] that
  identifies $\Fun(\mathcal{I}, \mathcal{C})$ with the functors that
  preserve cocartesian morphisms over inert morphisms in $\Dop$. \qed
\end{cor}
Under this equivalence, functors $\mathcal{I} \to
\SPANp(\mathcal{C})$ over $\Dop$ that preserve cocartesian
morphisms over inert maps correspond to functors $F \colon \mathcal{I} \to
\mathcal{C}$ such that for all  $x \in \mathcal{I}$ over $[n] \in
\Dop$, the morphism
\[ F(x) \to F(x_{0,1})\times_{F(x_{1,1})} \cdots
  \times_{F(x_{n-1,n-1})} F(x_{n-1,n})\]
is an equivalence, where $x \to x_{i,j}$ denotes the cocartesian
morphism over $\rho_{i,j}$. In particular, we have:
\begin{cor}\label{AlgSpanSeg} 
  Let $\mathcal{O} \to \Dop$ be a (generalized) non-symmetric \iopd{}
  and $\mathcal{C}$ an \icat{} with pullbacks. Then there is a natural
  equivalence of \icats{}
  \[ \Seg_{\mathcal{O}}(\mathcal{C}) \simeq
    \Alg_{\mathcal{O}}(\SPANp(\mathcal{C})).\]
\end{cor}

\begin{propn}\label{SegOfibAlg}
  For $X \in \mathcal{C}$, the fibre of
  $\txt{ev}_{[0]} \colon \Seg_{\mathcal{O}}(\mathcal{C}) \to
  \Fun(\mathcal{O}_{0},\mathcal{C})$  at the constant functor with
  value $X$ is naturally equivalent to $\Alg_{\mathcal{O}}(\mathcal{C}_{/X,X}^{\otimes})$.
\end{propn}
\begin{proof}
  Combining the equivalences of \cref{AlgSpanSeg} and
  \cref{AlgDopC}, we can identify $\Seg_{\mathcal{O}}(\mathcal{C}) \to
  \Fun(\mathcal{O}_{0},\mathcal{C})$ with the functor
  $\Alg_{\mathcal{O}}(\SPANp(\mathcal{C})) \to
  \Alg_{\mathcal{O}}(\Dop_{\mathcal{C}})$; the constant functor
  $\mathcal{O}_{0} \to \mathcal{C}$ corresponds to the composite
  $\mathcal{O} \to \Dop \to \Dop_{\mathcal{C}}$ where the second
  morphism is the associative algebra in $\Dop_{\mathcal{C}}$
  associated to $X \in \mathcal{C}$. Since
  $\Alg_{\mathcal{O}}(\blank)$ preserves limits, the fibre we want is
  given by the pullback square
  \[
    \begin{tikzcd}
      {} &[-3.45em] \Alg_{\mathcal{O}}(\mathcal{C}^{\otimes}_{/X,X}) \arrow{d}
      \arrow{r} & \Alg_{\mathcal{O}}(\SPANp(\mathcal{C})) \arrow{d}
      \\
      * \simeq & \Alg_{\mathcal{O}}(\Dop) \arrow{r} & \Alg_{\mathcal{O}}(\Dop_{\mathcal{C}}),
    \end{tikzcd}
  \]
  as required.
\end{proof}

\section{Quasi-Unital Category Objects}\label{sec:qucat}
Our goal in this section is to prove that ``having identities'' is a
property of a category object. We will prove this by using the results
of the previous section to reduce to the case of associative algebras,
which has already been proved by Lurie.  We begin by recalling Lurie's
result, which requires introducing some notation:

\begin{defn}
  Let $\Dinjop$ denote the subcategory of $\Dop$ containing only the
  injective maps; this is a \nsiopd{}, and its algebras are non-unital
  associative algebras.  If $\mathcal{O}$ is a \gnsiopd{}, we
  write $\Algnu(\mathcal{O})$ for the \icat{}
  $\Alg_{\Dinjop}(\mathcal{O})$ of non-unital associative algebras in
  $\mathcal{O}$. The inclusion $j \colon \Dinjop \to \Dop$ is a
  morphism of \nsiopds{}, and induces the expected forgetful functor
  $j^{*}\colon \Alg(\mathcal{O}) \to \Algnu(\mathcal{O})$.
\end{defn}

\begin{defn}
  Let $\mathcal{V}^{\otimes} \to \Dop$ be a monoidal \icat{}. If $A$
  is a non-unital associative algebra in $\mathcal{V}$,
  a \emph{quasi-unit} for $A$ is a 
  morphism $u \colon I \to A$ such that the composite
  \[ A \simeq I \otimes A \xto{u \otimes A} A \otimes A \xto{m} A,\]
  where $m$ is the algebra multiplication, is equivalent to $\id_{A}$,
  and similarly for the map with $u$ on the other side. We say that
  a non-unital algebra $A$ is \emph{quasi-unital} if there exists a
  quasi-unit for $A$. If $A$ and $B$
  are quasi-unital algebras, we say that a morphism $f \colon A \to B$
  in $\Algnu(\mathcal{V})$ is \emph{quasi-unital} if $f \circ u$ is a
  quasi-unit for $B$, where $u \colon I \to A$ is a quasi-unit for
  $A$.  
\end{defn}

\begin{warning}
  We emphasize that being quasi-unital is a \emph{property} of a
  non-unital algebra. In particular, the data of a quasi-unit is
  \emph{not} part of the structure of a quasi-unital algebra, we are
  merely asserting that it is possible to choose one. 
\end{warning}

\begin{defn}
  Let $\Algqu(\mathcal{V})$ denote the subcategory of
  $\Algnu(\mathcal{V})$ whose objects are the quasi-unital algebras,
  and whose morphisms are the quasi-unital ones.
\end{defn}

\begin{thm}[Lurie, \cite{HA}*{Theorem 5.4.3.5}]\label{qualg}
  If $\mathcal{V}$ is a monoidal \icat{}, then the functor
  $j^{*} \colon \Alg(\mathcal{V}) \to \Algnu(\mathcal{V})$ induces an equivalence
  \[ \Alg(\mathcal{V}) \isoto \Algqu(\mathcal{V})\] onto the
  quasi-unital subcategory.
\end{thm}

For the rest of this section we fix an \icat{} $\mathcal{C}$ with
finite limits. We can then extend the definitions above to category
objects in $\mathcal{C}$:
\begin{defn}
  A \emph{non-unital category object} in $\mathcal{C}$ is a Segal
  $\Dinjop$-object, \ie{} a functor $\Dinjop \to \mathcal{C}$
  satisfing the same limit condition as a category object. We
  write $\Catnu(\mathcal{C})$ for the \icat{}
  $\Seg_{\Dinjop}(\mathcal{C})$ of non-unital category objects.
\end{defn}

\begin{defn}
  Let $X \colon \Dinjop \to \mathcal{C}$ be a non-unital category
  object. A \emph{quasi-unit} for $X$ is a 
  commutative diagram
  \[
    \begin{tikzcd}
      X_{0} \arrow{rrr}{u} \arrow[equals]{dr} \arrow[equals]{drr} & &
      & X_{1} \arrow[crossing over]{dl}{d_{0}} \arrow{dll}[above left]{d_{1}} \\
       & X_{0} & X_{0}
    \end{tikzcd}
  \]
  such that the composite
  \[ X_{1} \simeq X_{0} \times_{X_{0}} X_{1} \xto{u \times_{X_{0}}
      X_{1}} X_{1} \times_{X_{0}} X_{1} \simeq X_{2} \xto{d_{1}}
    X_{1} \] is equivalent to the identity, and similarly for the
  morphism with $u$ on the other side.  We say a non-unital category
  object $X$ is \emph{quasi-unital} if there exists a quasi-unit for $X$.
\end{defn}

\begin{remark}
  A non-unital category object $X$  in $\mathcal{C}$ is quasi-unital \IFF{} $X$ is
  quasi-unital when viewed as a non-unital associative algebra in
  $\mathcal{C}_{/X_{0},X_{0}}^{\otimes}$. Any two quasi-units are
  therefore equivalent, by \cite{HA}*{Remark 5.4.3.2}.
\end{remark}

\begin{remark}\label{cartqunital}
  The functor
  $\txt{ev}_{[0]} \colon \Catnu(\mathcal{C}) \to \mathcal{C}$ is a
  cartesian fibration by \cref{Seg0cart}.  Suppose $X$ is a non-unital
  category object and $u \colon X_{0} \to X_{1}$ is a quasi-unit for
  $X$. For $f \colon Y \to X_{0}$ in $\mathcal{C}$, let $f^{*}X \to X$
  denote the cartesian morphism in $\Catnu(\mathcal{C})$ over
  $f$. Then we have a diagram
  \[
    \begin{tikzcd}
      Y \arrow{r}{f} \arrow[dashed]{d}{f^{*}u} & X_{0} \arrow{d}{u} \\
      (f^{*}X)_{1} \arrow{r} \arrow{d} & X_{1} \arrow{d} \\
      Y_{0} \times Y_{0} \arrow{r}  & X_{0} \times X_{0},
    \end{tikzcd}
  \]
  where the morphism $f^{*}u$ exists since the bottom square is
  cartesian. The morphism $f^{*}u$ is then a quasi-unit for
  $f^{*}X$. In other words, if $X$ is quasi-unital then so is $f^{*}X$
  for any $f \colon Y \to X_{0}$.
\end{remark}

\begin{defn}
  Suppose $X$ and $Y \in \Catnu(\mathcal{C})$ are quasi-unital. A
  morphism $\phi \colon X \to Y$ is \emph{quasi-unital} if there
  exists a commutative diagram
  \[
    \begin{tikzcd}
      X_{0} \arrow{ddr}{\Delta} \arrow{rrr}{\phi_{0}} \arrow{drr}{u} &
      & & Y_{0} \arrow{ddr}[below left,near start]{\Delta} \arrow{drr}{v} \\
           &  & X_{1} \arrow{dl}{(d_{0},d_{1})}\arrow[crossing
           over]{rrr}[below,near start]{\phi_{1}} & & & Y_{1} \arrow{dl}{(d_{0},d_{1})}\\
       & X_{0} \times X_{0} \arrow{rrr}{\phi_{0} \times \phi_{0}}& & & Y_{0} \times Y_{0},
    \end{tikzcd}
  \]
  where $u$ and $v$ are quasi-units for $X$ and $Y$, respectively. We
  write $\Catqu(\mathcal{C})$ for the subcategory of
  $\Catnu(\mathcal{C})$ containing the quasi-unital objects and the
  quasi-unital morphisms between them.
\end{defn}

\begin{remark}\label{quiffcartqu}
  From \cref{cartqunital} and the uniqueness of quasi-units we see
  that a morphism $\phi \colon X \to Y$ is quasi-unital \IFF{}
  $X \to \phi_{0}^{*}Y$ is quasi-unital. Moreover, the latter is
  quasi-unital \IFF{} it corresponds to a quasi-unital morphism
  between non-unital algebras in
  $\mathcal{C}_{/X_{0},X_{0}}^{\otimes}$.
\end{remark}

\begin{propn}
  Suppose $p \colon \mathcal{E} \to \mathcal{B}$ is a cartesian
  fibration, and $\mathcal{E}_{0}$ is a subcategory of $\mathcal{E}$
  such that
  \begin{enumerate}[(i)]
  \item for $x \in \mathcal{E}_{0}$ and $f \colon b \to p(x)$ in
    $\mathcal{B}$, the cartesian morphism $f^{*}x \to x$ lies in
    $\mathcal{E}_{0}$.
  \item if $x$ and $y$ are objects of $\mathcal{E}_{0}$ then a
    morphism $\phi \colon x \to y$ in $\mathcal{E}$ lies in
    $\mathcal{E}_{0}$ \IFF{} $x \to p(\phi)^{*}y$ lies in $\mathcal{E}_{0}$.
  \end{enumerate}
  Then $p|_{\mathcal{E}_{0}} \colon \mathcal{E}_{0} \to \mathcal{B}$
  is a cartesian fibration, and a morphism in $\mathcal{E}_{0}$ is
  cartesian \IFF{} its image in $\mathcal{E}$ is cartesian.
\end{propn}
\begin{proof}
  Given $x \in \mathcal{E}_{0}$ and $f \colon b \to p(x)$ we must show
  that the cartesian morphism $f^{*}x \to x$ in $\mathcal{E}$ is
  cartesian when viewed as a morphism in $\mathcal{E}_{0}$. For $y \in
  \mathcal{E}_{0}$ we have a commutative diagram
  \[
    \begin{tikzcd}
      \Map_{\mathcal{E}_{0}}(y, f^{*}x) \arrow{r} \arrow{d} &
      \Map_{\mathcal{E}_{0}}(y,x) \arrow{d} \\
      \Map_{\mathcal{E}}(y, f^{*}x) \arrow{r} \arrow{d} &
      \Map_{\mathcal{E}}(y,x) \arrow{d} \\
      \Map_{\mathcal{B}}(p(y), b) \arrow{r} & \Map_{\mathcal{B}}(p(y),p(x)).
    \end{tikzcd}
  \]
  Here the top square is cartesian by assumption (ii) and the bottom
  square is cartesian since $f^{*}x \to x$ is cartesian. It follows
  that the composite square is cartesian, which completes the proof.
\end{proof}

Combined with \cref{quiffcartqu} and \cref{Seg0cart}, this gives:
\begin{cor}\label{Catqucart}
  The functor $\txt{ev}_{[0]}\colon \Catqu(\mathcal{C}) \to
  \mathcal{C}$ is a cartesian fibration; a morphism in
  $\Catqu(\mathcal{C})$ is cartesian \IFF{} its image in
  $\Catnu(\mathcal{C})$ is cartesian. \qed
\end{cor}

\begin{thm}\label{qucat} 
  The functor $j^{*} \colon \Cat(\mathcal{C}) \to \Catnu(\mathcal{C})$ induces
  an equivalence
  \[ \Cat(\mathcal{C}) \isoto \Catqu(\mathcal{C})\] onto the
  quasi-unital subcategory.
\end{thm}
\begin{proof}
  We have a commutative triangle
  \[
    \begin{tikzcd}
      \Cat(\mathcal{C}) \arrow{dr}[below left]{i_{\Dop}^{*}} \arrow{rr}{j^{*}} & &
      \Catnu(\mathcal{C}) \arrow{dl}{i_{\Dinjop}^{*}} \\
      & \mathcal{C}.
    \end{tikzcd}
  \]
  Here the diagonal functors are both cartesian fibrations by
  \cref{Seg0cart}. We claim that $j^{*}$ also preserves cartesian
  morphisms. Using the description of the cartesian morphisms in
  \cref{Seg0cart}, this amounts to observing that the canonical
  natural transformation $j^{*}i_{\Dop,*} \to i_{\Dinjop,*}$ is
  clearly an equivalence. The functor $j^{*}$ obviously factors
  through the subcategory $\Catqu(\mathcal{C})$, so by
  \cref{Catqucart} we get a commutative triangle
  \[
    \begin{tikzcd}
      \Cat(\mathcal{C}) \arrow{dr}[below left]{i_{\Dop}^{*}} \arrow{rr}{j^{*}_{\txt{qu}}} & &
      \Catqu(\mathcal{C}) \arrow{dl}{i_{\Dinjop,\txt{qu}}^{*}} \\
      & \mathcal{C},
    \end{tikzcd}
  \]
  where the diagonal functors are cartesian fibrations, and the
  horizontal functor preserves cartesian morphisms. To prove that
  $j^{*}_{\txt{qu}}$ is an equivalence, it therefore suffices to prove
  that for every object $X \in \mathcal{C}$ the functor
  \[ \Cat(\mathcal{C})_{X} \to \Catqu(\mathcal{C})_{X}\]
  on fibres over $X$ is an equivalence. But by \cref{quiffcartqu} and
  \cref{SegOfibAlg} we can identify this with the restriction of
  $j^{*} \colon \Alg(\mathcal{C}^{\otimes}_{/X,X}) \to
  \Algnu(\mathcal{C}^{\otimes}_{/X,X})$ to a functor
  \[ \Alg(\mathcal{C}^{\otimes}_{/X,X}) \to
    \Algqu(\mathcal{C}^{\otimes}_{/X,X}),\]
  which is an equivalence by \cref{qualg}.  
\end{proof}

We can inductively define the \icat{} of \emph{$n$-uple category
  objects} in $\mathcal{C}$ as
\[\Cat^{n}(\mathcal{C}) :=
\Cat(\Cat^{n-1}(\mathcal{C}));\] this corresponds to a full subcategory
of $\Fun(\simp^{n,\op}, \mathcal{C})$. Similarly, we can define
$\Catnu^{n}(\mathcal{C})$ and $\Catqu^{n}(\mathcal{C})$. Applying
\cref{qucat} inductively, we get:
\begin{cor}\label{nuple}
  The restriction functor $\Cat^{n}(\mathcal{C}) \to
  \Catnu^{n}(\mathcal{C})$ factors through an equivalence
  \[ \Cat^{n}(\mathcal{C}) \isoto \Catqu^{n}(\mathcal{C}).\]
\end{cor}

The $n$-uple category objects in the \icat{} $\mathcal{S}$ of spaces
are known as $n$-uple Segal spaces. By imposing constancy conditions,
we can restrict to the class of \emph{$n$-fold Segal spaces}
\cite{BarwickThesis}, which model (the algebraic structure of)
$(\infty,n)$-categories. This notion makes sense more generally:
\begin{defn}
  A 1-fold Segal object in $\mathcal{C}$ is a category object. For
  $n > 1$, an $n$-fold Segal object is an $n$-uple category object $X$
  such that
  $X_{0,\bullet,\ldots,\bullet} \colon \simp^{n-1,\op} \to
  \mathcal{C}$ is constant and $X_{i,\bullet,\ldots,\bullet}$ is an
  $(n-1)$-fold Segal object for all $i > 0$.
\end{defn}
If we write $\Seg^{n}(\mathcal{C})$ for the full subcategory of
$\Cat^{n}(\mathcal{C})$ spanned by the $n$-fold Segal objects, and
$\Seg^{n}_{\txt{nu}}(\mathcal{C})$ for the analogously defined
non-unital $n$-fold Segal objects, restricting the equivalence of \cref{nuple} gives:
\begin{cor}\label{Segnnueq}
  The restriction functor $\Seg^{n}(\mathcal{C}) \to
  \Segnu^{n}(\mathcal{C})$ factors through an equivalence
  \[ \Seg^{n}(\mathcal{C}) \isoto \Segqu^{n}(\mathcal{C}),\]
  where $\Segqu^{n}(\mathcal{C})$ is the subcategory of
  non-unital $n$-fold Segal spaces that are quasi-unital (when viewed
  as $n$-uple category objects) and quasi-unital morphisms between
  them.
\end{cor}

The \icat{} $\Cat_{(\infty,n)}$ of $(\infty,n)$-categories can be
described as the full subcategory
$\CSS^{n}(\mathcal{S}) \subseteq \Seg^{n}(\mathcal{S})$ consisting of
the \emph{complete} $n$-fold Segal spaces. The equivalence of \cref{Segnnueq} thus
identifies $\Cat_{(\infty,n)}$ this with a certain full subcategory of
$\Segqu^{n}(\mathcal{S})$. This full subcategory can also be described
without reference to this equivalence, as follows:

\begin{defn}
  If $X$ is a non-unital Segal space, we write $X(x,y)$ for the fibre
  of $(d_{1},d_{0}) \colon X_{1} \to X_{0} \times X_{0}$ at
  $(x,y)$; we refer to the points of $X(x,y)$ as \emph{morphisms} from
  $x$ to $y$, and denote these in the usual way. If $X$ is
  quasi-unital with quasi-unit $u \colon X_{0}\to X_{1}$ we say that a morphism $f \colon
  x \to y$ is an \emph{equivalence} in $X$ if there exists $g \colon y
  \to x$ such that $g \circ f \simeq u(x)$ and $f \circ g \simeq
  u(y)$. Let $X_{1}^{\txt{eq}}$ denote the subspace of $X_{1}$
  containing those components that correspond to equivalences. It is
  clear from the definition of a quasi-unit that $u$ restricts to a
  map $X_{0}\to X_{1}^{\txt{eq}}$, and we say that
  $X$ is \emph{complete} if this map is an equivalence in
  $\mathcal{S}$.
\end{defn}

\begin{remark}\label{rmk:eqcenu}
  More generally, we can define an equivalence in a non-unital Segal
  space $X$ to be a morphism $f \colon x \to y$ such that the
  morphisms
  \[ f^{*} \colon X(y,z) \to X(x,z), \quad f_{*} \colon X(z,x) \to
    X(z,y) \]
  given by composition with $f$ are equivalences in $\mathcal{S}$ for
  all $z \in X_{0}$. It is easy to see that this is equivalent to the
  previous definition when $X$ is quasi-unital. Moreover, if $X$ is
  quasi-unital then $X$ is complete \IFF{} either (or both) of the
  face maps $X_{1}^{\txt{eq}} \to X_{0}$ is an equivalences, since the
  quasi-unit is always a section of both. Thus we can characterize the
  complete quasi-unital Segal spaces without explicitly referring to
  the quasi-unit.
\end{remark}

\begin{defn}
  Suppose $X \colon \Dinj^{n,\op} \to \mathcal{S}$ is an $n$-fold
  quasi-unital Segal space. Then $X$ is \emph{complete} if
  \begin{itemize}
  \item the quasi-unital Segal space $X_{\bullet,0,\ldots,0}$ is
    complete,
  \item the quasi-unital $(n-1)$-fold Segal space
    $X_{1,\bullet,\ldots,\bullet}$ is complete.
  \end{itemize}
  Equivalently, $X$ is complete if the $n$ quasi-unital Segal spaces
  of the form $X_{1,\ldots,1,\bullet,0,\ldots,0}$ are all complete. We
  write $\CSSqu^{n}(\mathcal{S})$ for the full subcategory of
  $\Segqu^{n}(\mathcal{S})$ spanned by the complete quasi-unital
  $n$-fold Segal spaces.
\end{defn}

\begin{cor}\label{CSSnqueq}
  The restriction functor $\Seg^{n}(\mathcal{S}) \to
  \Segnu^{n}(\mathcal{S})$ restricts to an equivalence
  \[ \CSS^{n}(\mathcal{S}) \to \CSSqu^{n}(\mathcal{S}). \]
\end{cor}
\begin{proof}
  This is immediate from \cref{Segnnueq}, since the definition of
  completeness for a quasi-unital $n$-fold Segal space clearly
  restricts to the usual notion of completeness for an $n$-fold Segal
  space.
\end{proof}

\begin{warning}\label{warn:Harpaz}
  Our notion of a quasi-unital Segal space is not quite the same as
  that studied by Harpaz~\cite{HarpazQUnital}. Let us say
  that a morphism $u \colon x \to x$ in a non-unital Segal space $X$
  is a \emph{quasi-identity} if there are equivalences
  $f \circ u \simeq f$ for $f$ in $X(x,y)$ and $u \circ g \simeq g$
  for $g$ in $X(y,x)$, natural in $f$ and $g$ (more precisely, for all
  $y$ the maps $u_{*} \colon X(y,x) \to X(y,x)$ and
  $u^{*} \colon X(x,y) \to X(x,y)$ are equivalent to the corresponding
  identities). We can then call $X$ \emph{weakly quasi-unital}
  if for every object $x$ there exists a quasi-identity $u \colon x
  \to x$; these objects correspond to those called quasi-unital in
  \cite{HarpazQUnital}. In general this seems to be a strictly
  weaker notion than ours, which essentially requires there to exist a
  choice of quasi-identities $u_{x} \colon x \to x$ that is continuous
  in $x \in X_{0}$; we expect that this stronger condition is needed
  to get an equivalence between Segal spaces and quasi-unital Segal
  spaces (whereas Harpaz only gets an equivalence between the
  subcategories of complete objects). 
\end{warning}

\begin{remark}
  We can define a non-unital Segal space $X$ to be \emph{complete} if
  the two face maps $X_{1}^{\text{eq}} \to X_{0}$ are both
  equivalences, following \cref{rmk:eqcenu} and
  \cite{HarpazQUnital}*{\S 3}; it is easy to see
  that if $X$ is complete then it is also weakly
  quasi-unital. It follows from Harpaz's work that every
  complete non-unital Segal
  space is in the image of the restriction functor from complete Segal
  spaces, and thus is in particular quasi-unital in our
  sense. Similarly, a morphism of complete non-unital Segal spaces is
  quasi-unital \IFF{} it preserves quasi-identities. 
\end{remark}


\begin{bibdiv}
\begin{biblist}
\bib{BarwickThesis}{book}{
  author={Barwick, Clark},
  title={$(\infty ,n)$-{C}at as a closed model category},
  note={Thesis (Ph.D.)--University of Pennsylvania},
  date={2005},
}

\bib{BarwickMackey}{article}{
  author={Barwick, Clark},
  title={Spectral {M}ackey functors and equivariant algebraic $K$-theory ({I})},
  journal={Adv. Math.},
  volume={304},
  date={2017},
  pages={646--727},
  eprint={arXiv:1404.0108},
  year={2014},
}

\bib{BurroniTCat}{article}{
  author={Burroni, Albert},
  title={$T$-cat\'{e}gories (cat\'{e}gories dans un triple)},
  language={French},
  journal={Cahiers Topologie G\'{e}om. Diff\'{e}rentielle},
  volume={12},
  date={1971},
  pages={215--321},
}

\bib{BenabouBicat}{article}{
  author={B\'{e}nabou, Jean},
  title={Introduction to bicategories},
  conference={ title={Reports of the Midwest Category Seminar}, },
  book={ publisher={Springer, Berlin}, },
  date={1967},
  pages={1--77},
}

\bib{CruttwellShulman}{article}{
  author={Cruttwell, G. S. H.},
  author={Shulman, Michael A.},
  title={A unified framework for generalized multicategories},
  journal={Theory Appl. Categ.},
  volume={24},
  date={2010},
  pages={No. 21, 580--655},
}

\bib{DyckerhoffKapranovTwoSeg}{book}{
  eprint={arXiv:1212.3563},
  author={Dyckerhoff, Tobias},
  author={Kapranov, Mikhail},
  title={Higher Segal spaces},
  series={Lecture Notes in Mathematics},
  volume={2244},
  publisher={Springer, Cham},
  date={2019},
}

\bib{FioreGambinoKock}{article}{
  author={Fiore, Thomas M.},
  author={Gambino, Nicola},
  author={Kock, Joachim},
  title={Monads in double categories},
  journal={J. Pure Appl. Algebra},
  volume={215},
  date={2011},
  number={6},
  pages={1174--1197},
}

\bib{GaitsgoryRozenblyum1}{book}{
  author={Gaitsgory, Dennis},
  author={Rozenblyum, Nick},
  title={A study in derived algebraic geometry. Vol. I. Correspondences and duality},
  series={Mathematical Surveys and Monographs},
  volume={221},
  publisher={American Mathematical Society, Providence, RI},
  date={2017},
  note={Available from \url {http://www.math.harvard.edu/~gaitsgde/GL/}.},
}

\bib{enr}{article}{
  author={Gepner, David},
  author={Haugseng, Rune},
  title={Enriched $\infty $-categories via non-symmetric $\infty $-operads},
  journal={Adv. Math.},
  volume={279},
  pages={575--716},
  eprint={arXiv:1312.3178},
  date={2015},
}

\bib{freepres}{article}{
  author={Gepner, David},
  author={Haugseng, Rune},
  author={Nikolaus, Thomas},
  title={Lax colimits and free fibrations in $\infty $-categories},
  eprint={arXiv:1501.02161},
  journal={Doc. Math.},
  volume={22},
  date={2017},
  pages={1225--1266},
}

\bib{GrandisPareLimits}{article}{
  author={Grandis, Marco},
  author={Par\'e, Robert},
  title={Limits in double categories},
  journal={Cahiers Topologie G\'{e}om. Diff\'{e}rentielle Cat\'{e}g.},
  volume={40},
  date={1999},
  number={3},
  pages={162--220},
}

\bib{HarpazQUnital}{article}{
  author={Harpaz, Yonatan},
  title={Quasi-unital $\infty $-categories},
  journal={Algebr. Geom. Topol.},
  volume={15},
  date={2015},
  number={4},
  pages={2303--2381},
}

\bib{spans}{article}{
  author={Haugseng, Rune},
  title={Iterated spans and classical topological field theories},
  journal={Math. Z.},
  volume={289},
  issue={3},
  pages={1427--1488},
  date={2018},
  eprint={arXiv:1409.0837},
}

\bib{nmorita}{article}{
  author={Haugseng, Rune},
  title={The higher {M}orita category of $E_{n}$-algebras},
  date={2017},
  eprint={arXiv:1412.8459},
  journal={Geom. Topol.},
  volume={21},
  issue={3},
  pages={1631--1730},
}

\bib{JoyalKockUnits}{article}{
  author={Joyal, Andr\'{e}},
  author={Kock, Joachim},
  title={Coherence for weak units},
  journal={Doc. Math.},
  volume={18},
  date={2013},
  pages={71--110},
}

\bib{LeinsterGenEnr}{article}{
  author={Leinster, Tom},
  title={Generalized enrichment of categories},
  note={Category theory 1999 (Coimbra)},
  journal={J. Pure Appl. Algebra},
  volume={168},
  date={2002},
  number={2-3},
  pages={391--406},
}

\bib{LeinsterHigherOpds}{book}{
  author={Leinster, Tom},
  title={Higher operads, higher categories},
  series={London Mathematical Society Lecture Note Series},
  volume={298},
  publisher={Cambridge University Press},
  place={Cambridge},
  date={2004},
  pages={xiv+433},
}

\bib{HTT}{book}{
  author={Lurie, Jacob},
  title={Higher Topos Theory},
  series={Annals of Mathematics Studies},
  publisher={Princeton University Press},
  address={Princeton, NJ},
  date={2009},
  volume={170},
  note={Available from \url {http://math.ias.edu/~lurie/}},
}

\bib{HA}{book}{
  author={Lurie, Jacob},
  title={Higher Algebra},
  date={2017},
  note={Available at \url {http://math.ias.edu/~lurie/}.},
}

\bib{LyubashenkoManzyukUnital}{article}{
  author={Lyubashenko, Volodymyr},
  author={Manzyuk, Oleksandr},
  title={Unital $A_{\infty }$-categories},
  conference={ title={Problems of Topology and Related Questions}, },
  book={ series={Proc. of Inst. of Mathematics NASU}, volume={3}, publisher={Inst. of Mathematics, Nat. Acad. Sci. Ukraine}, place={Kyiv}, },
  date={2006},
  pages={235--268},
  eprint={arXiv:0802.2885},
}

\bib{RezkCSS}{article}{
  author={Rezk, Charles},
  title={A model for the homotopy theory of homotopy theory},
  journal={Trans. Amer. Math. Soc.},
  volume={353},
  date={2001},
  number={3},
  pages={973--1007},
}

\bib{ShulmanFramed}{article}{
  author={Shulman, Michael},
  title={Framed bicategories and monoidal fibrations},
  date={2008},
  journal={Theory Appl. Categ.},
  volume={20},
  pages={No. 18, 650\ndash 738},
}

\bib{SteimleDeg}{article}{
  author={Steimle, Wolfgang},
  title={Degeneracies in quasi-categories},
  journal={J. Homotopy Relat. Struct.},
  volume={13},
  date={2018},
  number={4},
  pages={703--714},
}

\bib{TanakaUnital}{article}{
  author={Tanaka, Hiro Lee},
  title={Functors (between $\infty $-categories) that aren't strictly unital},
  journal={J. Homotopy Relat. Struct.},
  volume={13},
  date={2018},
  number={2},
  pages={273--286},
}
\end{biblist}
\end{bibdiv}
\end{document}